\documentclass[amssymb,12pt]{amsart}
\usepackage{latexsym}
\newtheorem{Theorem}{Theorem}
\newtheorem{Thm}{Theorem}
\newtheorem {Lm}[Thm]{Lemma}

\newtheorem {Cor} [Thm] {Corollary}
\newtheorem {Def}[Thm]{Definition}
\theoremstyle{definition}

\newtheorem {Example}{Example}

\renewenvironment{proof}{\noindent {\em Proof. }}{\hfill $\Box$}

\renewcommand {\phi} {\varphi}

\author{ David B. Wales}
\address{David B. Wales \\Department of Mathematics \\
Sloan Lab\\
Caltech\\
Pasadena, CA 91125\\
USA.}
\email{dbw@caltech.edu.}

\title{Classification of imprimitive irreducible finite subgroups of $O(7)$}

\date{\today}
\begin{document}

\begin{abstract}
This  work gives a  classification of imprimitive irreducible finite subgroups of the orthogonal group $O(7)$ plus the number of
conjugacy classes for each group.
\end{abstract}

\maketitle

\medskip
{\sc keywords:}  Orthogonal group, imprimitive linear group, quasiprimitve linear group

\medskip
{\sc AMS 2010 Mathematics Subject Classification:}
16Gxx, 17Bxx, 20Cxx

\section{Introduction}

Recently, Nicholas Katz asked if there was a classification of the finite irreducible imprimitive subgroups of $O(7)$.  He had encountered this question in studying the monodromy
groups attached to certain families of exponential sums.   He knew a priori they were irreducible subgroups of $S(O7)$ in its natural $7$-dimensional representation and
could deal with the finite primitive ones as they were known.  It would be natural to try to deal with the imprimitive ones in a similar way.  This paper provides the classification
of the imprimitive such groups.  It also corrects an error in \cite[Theorm $1$]{W} and \cite[Theorem $1$]{DW}. See Section $3$.

In Theorem \ref{Classif}
we   give a  classification of imprimitive irreducible finite subgroups of $O(7)$.  Because such a group  is imprimitive, there is a set of
 linearly independent non-trivial subspaces which span the space $R^7$ and which are permuted  by the group. See \cite[Definition $50.1$]{CR}. It is transitive because of the irreducibility.  As $7$ is a prime and the subspaces have the same dimension,
the subspaces must all have dimension 1.

\begin{Theorem}\label{Classif}  Suppose $G$ is a finite irreducible subgroup of $O(7) $ which is imprimitive.  Then $G$ is one of the groups listed below or a direct
product of $Z_2$ with one of these groups by adjoining the matrix $-I_7$.

\medskip
\noindent Case $1$.   The representation is quasi-primitive and $G$ is a simple group isomorphic to $PSL(2,7) $ as in \cite[Theorem $1$]{W} or \cite[Theorem $1$]{DW}.  The definition
of quasi-primitive is given at the beginning of section $3$.

\bigskip

 \noindent  If it is not quasi-primitive, there is a normal elementary Abelian subgroup, $A$,  of order $8$ or $64$.  We now distinguish
between the cases for $A$.  Furthermore, with one exception, $G$ is a semi-direct product of $A$ by $G/A$ which means it splits.

\medskip
\noindent Case $2$.  The finite subgroup $A$ has order $8$.  This means $G/A$ is a subgroup
of $PSL(3,2)\cong PSL(2,7)$.  Furthermore, the subgroup  $G/A$ of $PSL(3,2)$ must have order divisible by $7$ and must have a subgroup of index $7$.  The possible $G/A$ are either in the normalizer
of a Sylow $7$-group and have order $7$ or  $7\cdot3$, or $G/A$ is the full $PSL(3,2)$.  This means $G$ is an extension of $A$ of order $8$ by one of these groups.  They are all split except in the
 case of an extension by $PSL(2,3)$ in which case there are two such groups a split and a non-split group.
\medskip

\noindent Case $3$.  The finite subgroup $A$  has order $64$.  This means $G/A$ is a subgroup of
 $Sym(7)$.  The group $G/A$ must must have order a multiple of $7$ and must have a subgroup of index $7$.  The possible $G/A$ are either in the normalizer of a Sylow $7$-group in which case they have
 orders $7$, $7\cdot 2$, $7\cdot 3$,  or $7\cdot 42$ or $G/A$ is one of $PSL(3,2)$, $Alt(7)$ or $Sym(7)$.  In particular, $G$ is a split  extension of $A$ of order $64$ by one of these groups.
The groups in the normalizer of a Sylow $7$-group have a self-centralizing Sylow $7$-group.

\end{Theorem}
The numbers of conjugacy classes in each of these groups is given in Table $1$.

In section $2$ we give some general results and in section $3$ some preliminary results for $SO(7)$.  There is an Abelian group which enters the picture and in section $4$ we show it is of order
$8$ or $64$.  In section $5$ we give a general structure theorem which gives orders for the possible groups.  In section $6$ we give two examples which are relevant to
the work.  In section $7$ we describe a method to determine the numbers of conjugate classes.  In section $8$ we discuss which extensions are  split.  We end with section
$9$ where we give computations to obtain the numbers of conjugate classes.  These are tabulated in Table $1$.

\begin{section}{Some General Results}

We begin with a definition and immediate consequence for the group $O(n)$.  The group $O(n)$ is the group of all $n\times n$ matrices over the reals that are orthogonal.  This means for
a real matrix $A$ that $AA^{tr}=I$ where $A^{tr}$  is the transpose of the matrix A.  As $Det A=Det A^{tr}$ this means $DetA^2=1$ and so $Det A=\pm1$.  In particular there is a subgroup of $O(n)$ of index
2 called $SO(n)$ as there are orthogonal matrices of determinant $-1$ for all $n$.  For $n$ odd,  $-I_n$ is an orthogonal matrix with determinant $-1$  and
this means the subgroup $O(n)$ is the direct product $SO(n) \times \langle-I_n\rangle$.  As $7$ is odd, in considering subgroups
of $O(7)$ we  consider the subgroup $SO(7)$.

\begin{Lm}\label{index 2}
If $G$ is a subgroup of $O(7)$, then $G=\langle G_1, \langle -I_7 \rangle\rangle$ where $G_1$ is the subgroup of $G$ consisting of matrices of determinant $1$.  In particular, $G\cong G_1\times Z_2$.
\end{Lm}

We continue  with a  well-known result about finite irreducible linear groups of odd degree.

\begin{Thm}\label{orthog}  Suppose $G$ is a finite  irreducible linear group over the complex numbers of odd degree, $n$, for which all the character values are real.   Then $G$ can be conjugated
by a matrix in $GL(n,C)$ into
a subgroup of   $O(n)$.  This means the matrices can be conjugated so they are all real and orthogonal.
\end{Thm}
\begin{proof}  This uses the concept of Schur index as described in \cite[section $11$]{F}. By the hypothesis, the reals contain all the character values.  By the Schur index definition, there is a
degree $m$ extension of the reals over which the representation is realizable, and the minimal $m$ is called the Schur index.  The group can of course   be realized over the
complex numbers, a degree $2$ extension of the reals, $R$.  This means the Schur index over the reals is $1$ or $2$.  But if it is $2$, by \cite[$11.5$]{F},  the degree must be even as $2$ divides the degree.  However,
by assumption, the degree is odd and so it can be realized over the reals.  It then follows from \cite[section $8$ and Lemma $50$]{Collins}  that the matrices can be taken to be orthogonal.
\end{proof}

\end{section}

\begin{section}{Some Preliminary Results}
There are errors in \cite[Theorm $1$]{W} and \cite[Theorem $1$]{DW} in which the group $PSL(2,7) $ appears and is called primitive.  Indeed it is imprimitive.  The group is quasi-primitve but
not primitive.  Here quasi-primitive refers to a linear group in which the restriction to any normal subgroup has all irreducible constituents similar to one of them.  Here
there are no non-trivial normal subgroups and so it is quasi-primitive.  If a group is not quasi-primitive it is imprimitive by Clifford's Theorem, \cite[Theorem $2.3$]{Collins}.  Indeed
the classifications in \cite{W} and \cite{DW} are for  quasi-primitive irreducible finite  groups of degree $7$.  The group $PSL(2,7)$ is  induced from a non-trivial linear character of one
of the subgroups of order $24$ and index $7$ and so is imprimitive.  For the group $PGL(2,7)$ the  irreducible character of degree $7$ as listed in \cite{W} and \cite{DW} is indeed primitive as $PGL(2,7)$ does
not have a subgroup of index $7$ as the outer automorphism interchanges the two conjugate classes of subgroups of order $24$ in $PSL(2,7)$.

Suppose $G$ is a finite irreducible subgroup of the orthogonal group $SO(7)$ which is imprimitive but not quasi-primitive.  Let $X$ be the representation with character $\chi$.

There must be a normal Subgroup $A$ for which the constituents of $X$ restricted to $A$ are not
all isomorphic.  In this case by Clifford's theorem, there must be seven linear characters permuted transitively. Let $\lambda$ be one of them.   Using \cite[Corollary $24$]{Collins}  the restriction of
$X$ to $A$ is a set of diagonal matrices and so is Abelian.  By Frobenius reciprocity, \cite[Theorem $29$]{Collins},  $X$ is a constituent of $\lambda^G$.

\begin{Thm}\label{Aeven}  Suppose $G$ is an irreducible finite subgroup of $SO(7)$ which is imprimitive but  not quasi-primitive, then the subgroup $A$ described immediately above is an elementary
Abelian group of order $2^c$.  There is an element $\alpha$ in $G$ of order $7$ which acts non-trivially on $A$.
\end{Thm}
\begin{proof}  Let $B$ be a characteristic elementary Abelian subgroup of order $p^b$ for some prime $p$ of the group $A$.  Recall $X$ has real entries.  By Clifford's theorem,
the restriction to $B$ is a direct sum of linear spaces.  The diagonal entries are all real and $p$-th roots of $1$.  But for any $p>2$ this is impossible
as  non-trivial $p$-th roots have degree at least $2$ over the reals.  This shows $p=2$.
  The same argument as above shows if there is an element of order $4$, the irreducible constituents have degree at least $2$ as $i$ is not in the reals.  This means $A$ is an
elementary Abelian $2$-group.

There is an element, $\alpha$ which permutes the seven one dimensional spaces.  Fix one space and take the seven images of it by $\alpha ^i$ for $1\leq i \leq 7$.  Now the matrix of $\alpha$ with respect to this basis
has all coefficients $1$
except for the $(7,1)$ entry.  Because the group is in $SO(7)$, the final entry is $1$ as well.  This means $\alpha$ has order $7$.  It acts non-trivially by conjugation on $A$.

\end{proof}

\end{section}

\begin{section}{A bound for the Abelian Subgroup}
As in the last section we assume $G$ is a finite irreducible subgroup of $SO(7)$ which is imprimitive but not quasi-primitive.  As in Theorem \ref{Aeven} let $A$ be the Abelian subgroup.
\begin{Thm}\label{8 or 64}  The Abelian subgroup, $A$,  of $G$ has order $8$ or $64$.  In each case,  the element of order $7$ acts fixed point freely, ie it does not centralize any element of order $2$.
\end{Thm}
\begin{proof}  The maximum order of an Abelian group of $7\times 7 $ diagonal matrices with $\pm 1$ on the diagonal is $2^7$ and if the determinant is $1$ it is $2^6$.  If $A$, the Abelian subgroup, has order $2^a$ with $a\leq 6$ and $\alpha$ is an element of order $7$, restrict to
$<A,\alpha>$ where $\alpha $ is the element of order $7$ which does not centralize $A$.  This is a group with an Abelian normal subgroup of index $7$ and so by Ito's Theorem, \cite[53.18]{CR}, all irreducible representations have degree dividing $7$.  As it is non Abelian, this means
the restriction to $\langle A, \alpha \rangle$ is irreducible.  Suppose there is an element of order $14$.  If so, an element of order $2$ centralizes $A$ as $A$ is Abelian and also centralizes $\alpha$ and so is in the center.  But this means
 it is the scalar $-I_7$.  But it is unimodular and so this is false.  This means  the order is $8\cdot 7$ or $64\cdot 7$ and in both cases the element $\alpha$ of order $7$ acts fixed point freely.
\end{proof}
\end{section}

\begin{section}{Structure Theorem}

In this section we determine the preliminary structure for a finite imprimitive irreducible subgroup, $G$,  of $SO(7).$

\begin{Thm}\label{Structure Theorem}  Suppose $G$ is a finite irreducible imprimitive subgroup of $SO(7)$.  Then $G$ is one of the groups listed in
Theorem \ref{Classif} except when $G/A$ is even, the extension need not split in that it need not be a semi-direct product of $A$ by $G/A$.

\end{Thm}
\begin{proof}
If the representation is quasi-primitive it is the group $G$ as in \cite[Theorem $1$]{W} or \cite[Theorem $1$]{DW}.

In case $2$, $G/A$ must be a subgroup of $PSL(3,2)$.   The only subgroups of $PL(3,2)$ with a subgroup of index  $7$ are the whole automorphism group, $PSL(3,2)$, the Sylow $7$ group, and its normalizer of order $21$ which is nonabelian. This can
be determined by using the Atlas,  \cite[page $3$]{Atlas}.

In case $3$ the quotient by $A$ is a subgroup of $Sym(7) $ with a subgroup of index $7$.  The only such subgroups are the ones listed again by consulting \cite[page 10]{Atlas}.
\end{proof}

\begin{Cor}\label{trace rational}  If $G$ is a finite imprimitive irreducible subgroup of $O(7)$, the traces are all integers.

\end{Cor}

\begin{proof}  By restricting to $SO(7)$ the group is one of the ones listed in Theorem \ref{Structure Theorem}.  For these the traces are all integers as they are all monomial
matrices with entries $\pm 1$ and so the traces are all integers.  If $G$ is in $O(7)$ use the fact that $G=G_1 \langle -I_7\rangle$ by Lemma \ref{index 2}.
\end{proof}

\end{section}

\begin{section}{ Examples.}

\begin{Example}  Let $A$ be an elementary Abelian group of order $8$ and let $\alpha$ be an automorphism of $A$ of order $7$.  Notice  the  automorphism group
of $A$ is $PSL(3,2)$ which has order $7\cdot 3\cdot 8$ so $\alpha $ can be taken as any automorphism of $A$ of order $7$.  Let $G$ be the semi-direct product of $A$ by $\langle \alpha \rangle$.  Then
$|G|=8\cdot 7$.  Let $\chi$ be any non-trivial linear character of $A$ with representation $X$.  Then $\chi^G$ is an irreducible character as by Ito's Theorem the irreducible
character degrees are $1$ and $7$.  As the group is non Abelian, $\chi^G$ must be irreducible.  The matrices are all monomial
matrices.  An element of order $7$ acts as a seven cycle on the spaces and so the character value is $0$.  As $G/A$ is a cyclic group of order $7$, there are $7$ irreducible linear characters
on which the value on an element $\alpha$ of order $7$ is $1$.  Now column orthogonality forces the value of $\chi$ on an element of order $2$ to be $-1$ and shows again
that the value on $\alpha$ is $0$.  The sum of the squares
of the degrees so far is $56$ and so these are all the characters.  By Theorem \ref{orthog}, the matrices can be taken to be orthogonal and so in $O(7)$.
\end{Example}

\begin{Example} Let $A$ be either the elementary Abelian group of order $8$ or $ 64$ and let $\alpha$ be a non-trivial automorphism of order $7$ which does not
centralize any involution in $A$.  It is straightforward that there is such an automorphism as the automorphism group of an elementary Abelian group of order
$8$ contains such an element and in the case of $64$ a direct sum of two such groups can be taken.  Let $G$ be the semidirect product of $A$ by $\langle \alpha \rangle$.  This group
is non Abelian.  The induction of any non-trivial character of $A$ to $G$ is irreducible of degree $7$ and is in the reals and so conjugate to a subgroup of $O(7)$.
The irreducibility follows from Ito's theorem which has been used above and shows the irreducible representations have degree $1$ or $7$.  There are $7$ characters
of degree $1$ as $G/A$ has order $7$.  All other characters are of degree $7$ as the commutator subgroup is $A$ as the element of order $7$ acts fixed point
freely.  The constituents of the induced character cannot be trivial when restricted to $A$ by Frobenius reciprocity.

\end{Example}

\medskip
\noindent Consequence.  The groups appearing in Theorem \ref{Structure Theorem} all have irreducible representations of degree $7$ which are orthogonal.  This is because each
of the groups contains one of the groups  $A\cdot \langle \alpha \rangle$ and so the induced representations are all irreducible of degree $7$.  We can also determine the
irreducible character degrees.

\end{section}

\begin{section}{Numbers of Conjugate classes  of the groups.}

In this section and the next two we determine the numbers of conjugacy classes of each of the non quasi-primitive groups tabulated in Cases $2$ and $3$ of Theorem \ref{Classif}.  In fact we determine
the number of irreducible representations as this is the number of conjugate classes as these numbers are the same.  See \cite[Theorem $21$]{Collins}.  These are tabulated in Table~$1$.

For each of the groups listed in Theorem \ref{Structure Theorem} Case $2$ or Case $3$ there is an elementary Abelian subgroup of order $8$ or $64$.
There is a certain set of characters for $G$ which have $A$ in the kernel.  The number of these is the number of irreducible characters of $G/A$.  This number is readily calculated from $G/A$ possibly using the Atlas, \cite{Atlas}.

\begin{Def}\label{NFC}  Let  $NFC$  be the number of irreducible characters of $G/A$.
\end{Def}
\noindent  Here $NFC$ can be thought of as non faithful characters.

There is another class of characters which are faithful on $A$.

\begin{Def}\label{FC}  We let $FC$ be the number of characters which are faithful on $A$.
\end{Def}
\noindent  Here $FC$ can be thought of as faithful characters.

The  characters faithful on $A$ all have associated with them  seven  non-trivial linear characters of $A$ all conjugate by the element $\alpha$ of order seven.  Each is associated with a hyperplane of $A$ considered as a $GF(2)$ module.  For $A$ of order
$8$ there is one orbit of hyperplanes transitively permuted by $\alpha$.  For the group $A$ of order $64$ there are nine classes of hyperplanes each transitively permuted
by $\alpha$.

If $\chi$ is an irreducible character of $G$ which does not have $A$ in the kernel, by \cite[$9.11$]{F} the restriction to $A$ is a sum of seven nontrivial characters of $A$ permuted
transitively by $\alpha$.  Let $\eta$ be one of the non-trivial linear characters of $\chi | A$.

For fixed $\eta$ there is an inertial group  $I_\eta$ which fixes it.  This is the set of all $G$ in the group so that $\eta^G=\eta$.  The subgroup $I_\eta$ which
 stabilizes $\eta$,  is a subgroup of $G$ of index $7$ in our case.  Let $\theta$ be an irreducible character of $\chi |I_\eta$ for which $\theta | A$ contains $\eta $ as a constituent.   By \cite[$9.8$]{F} or \cite[Ex. $15$ page $69$]{Collins}, $\chi$
is $\theta^G$.
This uses the Mackey decomposition theorem which among other things gives a formula for $(\theta^G, \theta^G)$.  See
\cite[$9.8$]{F} or \cite[Ex. $15$ page $69$]{Collins}.
  If $y$ is in $I_\eta$, $y$ must fix the hyperplane, $K$, of $A$ which is the kernel for $\eta$.  For a fixed  $\eta$ we need to count the number of irreducibles of $I_\eta/K$ which represent $A/K$ by $-1$ where again  $K$ is the kernel of $\eta$.  If this number is $\gamma$, the number of characters, $\chi$ of $G$ which have $\eta$ as one of the constituents of $\chi |A$ is $\gamma$.  This
gives a formula for $FC$.  In the case of $A=64$, there are nine conjugacy classes of such $\eta$ and so $9\cdot \gamma$  such irreducibles.     For
the case of order $8$ there are $\gamma$.  The same characters are attained no matter which of the $\alpha$ conjugates of $\eta$ are taken, so we just take one such $\eta$.

\begin{Theorem}\label {Conjugate Classes}  If $G$ is one of the extensions in Theorem \ref{Structure Theorem} the  number of conjugate classes is $NFC + FC$.
\end{Theorem}

We have listed  these numbers for the extensions in Table $1$.

\end{section}

\begin{section}{Some non split cases}

Suppose $G$ is the group of order $64\cdot 7!$.  Here $A$ has order $64$ and is all unimodular diagonal matrices with entries $\pm1$.  Here $G/A \cong Sym(7)$.  We show this is a split extension of $A$ by $Sym(7)$.
If $G/A$ were non-split we use the  Gasch\"utz Theorem, \cite[$10.4$]{Asch} which says it is not split for a Sylow $2$-group.  However, we show that it is split for  a Sylow $2$-group.  Let $S$ be a Sylow $2$-group of $Sym(7)$ which has order $16$.
It can be taken to be the group generated by permutations $(12)(34)$, $(23)(56)$, $(56)$.  This group is isomorphic to $ Z_2\times Sym(4)$.  Consider these permutations  as $7$ by $7$ permutation matrices.  Except for
$(56)$ these are in $G$.  If we replace the permutation matrix for $(56)$ with its negative, it is unimodular and so is in $G$.  Now the group generated by these matrices with the matrix for $(56)$ replaced
with its negative, are in $G$ is also isomorphic to $Z_2\times Sym(4)$.  This means
 it  is a Sylow $2$-group for $G/A$.  This means the Sylow $2$-group splits over $A$ and
so $G$ is a split extension of $A$ by the Gasch\"utz Theorem .

The group of order $64\cdot {7! \over 2}$ is a subgroup of index $2$ and so it is split as well.  The same argument applies to the group of order $64\cdot 168$ as it is also a subgroup and has a Slow
subgroup of the same order, $8$.  This means
it also splits.  Here  the Sylow $2$-subgroups can be taken to be the same group of order $8$.

We next consider the extension of the elementary Abelian group $A$ of order $8$ by $PSL(3,2)$.  We use the computer algebra program GAP, \cite{GAP} to treat this case to show there is a split and a non-split
group.  The properties we have obtained so far show the possible groups are perfect and of order $1344$.  By using standard commands in GAP it can be determined there are only two such groups.  One
is a split extension and the other is a non-split extension.  Furthermore, there is a transitive permutation representation of degree $14$ for both of them.  The character tables for both groups
are given in GAP and they each have two irreducible representations of degree $7$ with integral characters.

These arguments show the remaining possibilities for a non-split extension are of order $64\cdot 7\cdot2$ and $64\cdot 7\cdot 42$.  Recall if $G/A$ is odd the groups split by elementary arguments or
certainly by the Gasch\"utz theorem.

As above, let $\eta$ be a linear character of $A$ with
kernel $K$.  Also as above, $I_{\eta} $ is a subgroup of index $7$ of $G$.
This means $I_\eta$ contains a Sylow $2$-subgroup of $G$.   As $G/A$ is non split, by the Gasch\"utz Theorem, \cite[$10.4$]{Asch}, $I_\eta$  is a non split extension of $A$.

Assume first $|G|$ has order $64\cdot 7\cdot 2$.  Here $G/A$ is a nonabelian group of order $14$.  As $G/A$ does not split from $A$, an element, $\tau$ of order $2$, not in $A$ must square to
a non-trivial element of $A$  or it would  split.  Let $K$ be the kernel of the linear character on $A$.  Pick a linear character, $\eta$  of $\chi| A$ which does not have $\tau^2$ in the kernel
which we take to be a hyperplane $K$.   We may have to take one of the conjugates under $\alpha$ for this but we retain the same notation.  This means $\tau^2 $ is $a$, a non trivial
element  of $A$.  Here $\eta (a)=-1$.  Here $I_\eta/A$  is a group of order $2$ in $I_\eta/A$ which we can take to be generated by $\tau$ mod $A$.  We can assume $\tau^2 = a$.  The map
$\phi$ of $I_\eta/K$ given by $\phi (\tau) = i$ with
kernel $K$ is a linear character of $I_\eta$ for which $\phi|A $ mod $K$  restricted to $A $ mod $K$  is $\eta$.   By \cite[$9.11$]{F}, $\chi $ is $\phi^G$.  The element $\tau$ inverts $\alpha$ and so has three $2$-cycles and one fixed point as a permutation in $Sym(7)$.  But this
means the character $\chi(\tau)= \pm i$ which is impossible for a real representation.  This means this case does not occur.  The same argument also shows the group of order $64\cdot 42$ must split as it has the same Sylow group with the element of order $2$ inverting the element of order $7$.

\begin{Cor}  The groups in Theorem \ref{Structure Theorem} must all be split extensions with the exception of the extension of the elementary Abelian group of order $8$ by $PSL(3,2)$ where there is
one non-split group and so Theorem \ref{Classif} holds.
\end{Cor}

This finishes the proof of Theorem \ref{Classif}.

\end{section}

\begin{section}{Specific calculations of numbers of conjugate classes.}
\bigskip
\begin{table}[ht]
\begin{center}
\begin{tabular}{|c|c|c|c|}
\hline
Group Order of $G$ & NFC & FC & Conjugate Classes\\
\hline
   &&&  \\
\hline
Case $2$ &&&  \\
\hline
$8\cdot 7$ & $7$ & $1$ & $8$  \\
\hline
$8\cdot 21$ & $5$ & $3$ &  $8$  \\
\hline
$8\cdot 168 $\quad{\rm split}  & $6$  &  $5$ &  $11 $   \\
\hline

$8\cdot 168 $ \quad{\rm non-split}  & $6$  &  $5$ &  $11 $   \\
\hline
  & & & \\
\hline

Case $3$ &&&  \\
\hline

$64 \cdot 7 $ & $ 7 $ &  $1\cdot 9=9$ &  $16$  \\
\hline
$64\cdot 14$  &  $5$  &  $2\cdot 9 = 18$ &  $ 23 $   \\
\hline
$64 \cdot 21$ &  $5$  &  $3\cdot 9=27$  & $32$  \\
\hline
$64 \cdot 42 $ & $10$  &  $6\cdot 9=54$ & $64$  \\
\hline
$64\cdot 168$  & $6$  &  $5\cdot 9=45$ & $51$   \\
\hline
$64\cdot {7!\over2}$ &  $9$  & $7\cdot 9=63$  &  $72$  \\
\hline
$64 \cdot 7!$  &  $15$ & $11\cdot 9=99$   &  $114$  \\

\hline
\end{tabular}
\end{center}
\caption{}{\textrm Number of conjugate classes of groups $G$ as listed in Theorem $1$.  For the middle columns
see Definitions \ref{NFC} and \ref {FC}. }
\end{table}

In this section we determine the number of conjugate classes for the groups in Theorem \ref{Classif} except for $PSL(3,2) $ which can be found in \cite[page 3]{Atlas}.
For some explanations for the computations for Table $1$ we refer to section $7$.  In each of the groups, $G$, there is an elementary Abelian
normal subgroup of order $8$ or $64$.  The number $NFC$ is the number of characters which are non-faithful on $A$.  This number
is the number of irreducible characters of $G/A. $  The first four rows of the table, labeled case $2$, have $|A|=8$.  For the first four rows in the table $G/A$ is $Z_7$, non-Abelian of order $21$, $PSL(3,2)$
and $PSL(3,2)$ again.
There are seven irreducible characters of a cyclic group of order $7$.  For the non-Abelian group of order $21$ there are five characters
two of degree $3$ and three of degree $1$.  For $PSL(3,2)$ there are six characters as seen from \cite[page $3$]{Atlas}.

For the remaining groups in the table labeled case $3$ the group $A$ has order $64$.  The groups
$G$ in order are $Z_7$, dihedral of order $14$, non-Abelian of order $21$, semidirect product of $Z_7$ and $Z_6$ acting fixed point freely on
$Z_7$, $PSL(3,2)$, $Alt(7)$, and $Sym(7)$.  As above there are seven characters of $Z_7$.  There are five characters of the dihedral group of
of order $14$, three of degree $2$ and two of degree $1$.  As above there are five characters of the group of order $21$.  For the group of
order $42$ there are four characters of degree $3$ and six of degree $1$.   For the group of order $168$ there are six characters as above.  For $Alt(7)$ there
are nine characters by \cite[page $10$]{Atlas}.  For $Sym(7)$ there are fifteen.  This can be seen from \cite[page $10$]{Atlas} using
the automorphism.  Here the two characters of degree ten are fused and each of the remaining ones extend to two characters.  This gives the values
in the second column.

In the third column we list the number of characters faithful on $A$.  Here let $\eta$ be a linear character of $A$  with kernel $K$ a subgroup of index
$2$ in $A$.  We need the number of characters of the  group  $I_\eta/K$ which when restricted to $A/K$ are $\eta$ which means they are non trivial.  The group $I_\eta$
is the inertial group of $\eta$ and was described in section $7$.  As it fixes $\eta$ under conjugation it also fixes $K$.
For the first case of $G=8\cdot 7$, $I_\eta/K$ is just $A/K$ and so there is only one character, as in the table.  For the group of order $21$, $I_\eta/K$ is a cyclic group
of order six so there are three.  For the case of $A/K \cong PSL(3,2)$ which splits, the group $I_\eta/K$ is $Z_2 \times Sym(4)$.  This can be seen from the \cite[page $3$]{Atlas} as
the subgroup of index seven.  There are in fact two such subgroups not conjugate in the automorphism group.  As $Sym(4)$ has five characters two of degree $3$, one of
degree $2$, and two of degree $1$. Similar reasoning works in the non-split case.  We also checked these values using GAP to give the number of conjugate
classes as $11$.   The group of order $21$ is
non-Abelian

Similar reasoning gives the values when $|A|$ has order $64$ which is case $3$.  Here, however, after finding the number of characters of $I_\eta/K$ which
are non-trivial on $A/K$ we multiply by $9$ as there are nine conjugacy classes of hyperplanes.   The same reasoning as above works when $G/A$ has order $7$ or $21$.  For $G/A$ of order $14$, $I_\eta/K$
is $Z_2\times Z_2$ and so the answer is $2$.  For the group of order $64\cdot 42$, $I_\eta/K$ is $Z_2 \times Z_6$.  This means there are six.
 For the case of $64\cdot 168$ the result is $5\cdot 9$ as above. This is a split extension and again $I_\eta/K\cong Z_2\times Sym(4)$.  For the case of $G/A\cong Alt(7)$,  $I_\eta/K \cong Alt(6) \times Z_2$.  Now
by \cite[page $5$]{Atlas} there are seven characters faithful on the $Z_2$.  For $G/A \cong Sym(7)$, $I_\eta / K $ is $Z_2\times Sym(6)$.  Using \cite[page $5$]{Atlas} there are eleven such characters as $Sym(6)$
has eleven characters, as the two characters of degree $8$ in $Alt(6)$ are fused and the remaining five extend in two ways.

\end{section}

\end{document}